\documentclass[11pt]{amsart}
\usepackage[english]{babel}
\usepackage{amssymb}
\usepackage{amsmath}
\usepackage{amsthm}
\usepackage{bbold}
\usepackage{bbm}
\usepackage{mathtools}
\usepackage{dsfont}
\usepackage[scr=rsfso]{mathalfa}
\newtheorem{thm}{Theorem}
\newtheorem{prop}{Proposition}

\newtheorem{lem}{Lemma}
\newtheorem{exam}{Example}

\newtheorem{definition}{Definition}

\numberwithin{equation}{section}
\newcommand{\abs}[1]{\lvert#1\rvert}
\newcommand{\norm}[1]{\lVert#1\rVert}

\newcommand{\oc}{\xrightarrow{\mathrm{o}}}	
\newcommand{\pc}{\xrightarrow{\mathrm{p}}}	
\newcommand{\unc}{\xrightarrow{\mathrm{un}}} 
\newcommand{\mc}{\xrightarrow{\mathrm{m}}} 
\newcommand{\umc}{\xrightarrow{\mathrm{um}}} 
\newcommand{\tc}{\xrightarrow{\tau}}	
\newcommand{\uoc}{\xrightarrow{\mathrm{uo}}}	
\newcommand{\nc}{\xrightarrow{\norm{\cdot}}}	
\newcommand{\wc}{\xrightarrow{\mathit{w}}}	
\DeclareMathAlphabet{\mathpzc}{OT1}{pzc}{m}{it}
\DeclareSymbolFont{bbold}{U}{bbold}{m}{n}
\DeclareSymbolFontAlphabet{\mathbbold}{bbold}
\def\one{\mathbbold{1}}

\DeclareMathOperator{\Span}{span}

\renewcommand{\le}{\leqslant}
\renewcommand{\ge}{\geqslant}

\begin{document}
	
\title{AMS Journal Sample}
	\author{Y. A. Dabboorasad$^{1,2}$}
	\address{$1$ Department of Mathematics, Islamic University of Gaza, P.O.Box 108, Gaza City, Palestine.}
	\email{yasad@iugaza.edu.ps, ysf\_atef@hotmail.com}

	\author{E. Y. Emelyanov$^2$}
	\author{M. A. A. Marabeh$^2$}
	\address{$^2$ Department of Mathematics, Middle East Technical University,  06800 Ankara, Turkey.}
	\email{yousef.dabboorasad@metu.edu.tr, eduard@metu.edu.tr, mohammad.marabeh@metu.edu.tr, m.maraabeh@gmail.com}
	
	\date{\today}
	
\title{$um$-Topology in multi-normed vector lattices}
\begin{abstract}
Let $\mathcal{M}=\{m_\lambda\}_{\lambda\in\Lambda}$ be a separating family of lattice seminorms on a vector lattice $X$,  
then $(X,\mathcal{M})$ is called a multi-normed vector lattice (or MNVL). 
We write $x_\alpha \xrightarrow{\mathrm{m}} x$ if $m_\lambda(x_\alpha-x)\to 0$ for all $\lambda\in\Lambda$. 
A net $x_\alpha$ in an MNVL $X=(X,\mathcal{M})$ is said to be unbounded $m$-convergent (or $um$-convergent) 
to $x$ if $\lvert x_\alpha-x \rvert\wedge u \xrightarrow{\mathrm{m}} 0$ for all $u\in X_+$. $um$-Convergence generalizes $un$-convergence \cite{DOT,KMT}
and $uaw$-convergence \cite{Zab}, and specializes $up$-convergence \cite{AEEM1} and $u\tau$-convergence \cite{DEM2}. 
$um$-Convergence is always topological, whose corresponding topology is called unbounded $m$-topology (or $um$-topology). 
We show that, for an $m$-complete metrizable MNVL $(X,\mathcal{M})$, the $um$-topology is metrizable iff $X$ has a countable 
topological orthogonal system. In terms of $um$-completeness, we present a characterization of MNVLs possessing both Lebesgue's 
and Levi's properties. Then, we characterize MNVLs possessing simultaneously the $\sigma$-Lebesgue and $\sigma$-Levi properties 
in terms of sequential $um$-completeness. Finally, we prove that any $m$-bounded and $um$-closed set is $um$-compact 
iff the space is atomic and has Lebesgue's and Levi's properties.
\end{abstract}
\maketitle

\section{Introduction and preliminaries}

Unbounded convergences have attracted many researchers (see for instance \cite{GaoX:14,Gao:14,GTX,EM16,DOT,Zab,KMT,AEEM1,LC,KV,KLT,GLX1,GLX2,Tay17,DEM2}. 
Unbounded convergences are well-investigated in vector and normed lattices (cf. \cite{DOT,GTX,KMT,Tr04,Wick77}). 
In the present paper, we also extend several previous results from \cite{DOT,GTX,KMT,Tr04,Wick77,Zab} to multi-normed setting.
This work is a continuation of \cite{DEM2}, in which unbounded topological convergence was studied in locally solid vector lattices.
	
For a net $x_\alpha$ in a  vector lattice $X$, we write $x_\alpha\oc x$ if $x_\alpha$ {\em converges to $x$ in order}. 
That is, there is a net $y_\beta$, possibly over a different index set, such that $y_\beta \downarrow 0$ and, for every
$\beta$, there exists $\alpha_\beta$ satisfying 	$\abs{x_\alpha-x}\le y_\beta$ whenever $\alpha\ge\alpha_\beta$. A net $x_\alpha$ in a vector lattice $X$ is 
{\em unbounded order convergent} ({\em uo-convergent}) to $x\in X$ if $\abs{x_\alpha-x}\wedge u\oc 0$ for every $u\in X_+$. We write $x_\alpha\uoc x$ in this case. 
Clearly, order convergence implies $uo$-convergence and they coincide for order bounded nets. For a measure space $(\Omega,\Sigma,\mu)$ and a sequence 
$f_n$ in $L_p(\mu)$ ($0\leq p\leq\infty$), $f_n\uoc 0$ iff $f_n\to 0$ almost everywhere \cite[Rem.3.4]{GTX}. It is known that almost everywhere convergence is not topological. Therefore, $uo$-convergence might not be topological in general. It was also shown recently that order convergence is never topological in infinite dimensional 
vector lattices \cite{DEM1}.
	
Let $(X,\lVert\cdot\rVert)$ be a normed lattice. For a net $x_\alpha$ in $X$, we write $x_\alpha\nc x$ if $x_\alpha$ converges to $x$ in norm. 
We say that $x_\alpha$ {\em unbounded norm converges} to $x$ ($x_\alpha$ {\em un-converges} to $x$ or $x_\alpha\unc x$) if $\abs{x_\alpha-x}\wedge u\nc 0$
for every $u\in X_+$. Clearly, norm convergence implies $un$-convergence. The $un$-convergence is topological, and the corresponding topology 
(which is known as {\em un-topology}) was investigated in \cite{KMT}. A net $x_\alpha$ {\em uaw-converges} to $x$ if 
$\abs{x_\alpha-x}\wedge u\wc 0$ for all $u\in X_+$, where ``$\mathit{w}$" stands for the weak convergence. Absolute weak convergence implies $uaw$-convergence. 
$uaw$-Convergence and $uaw$-topology were introduced and investigated in \cite{Zab}.

All topologies considered throughout this article are assumed to be Hausdorff. If a linear topology $\tau$ on a vector lattice $X$ has a base at zero consisting of solid sets, 
then the pair $(X,\tau)$ is called a {\em locally solid vector lattice}. Furthermore, if $\tau$ has base at zero consisting of convex-solid sets, then $(X,\tau)$ is called a 
{\em locally convex-solid vector lattice}. It is known that a linear topology $\tau$ on $X$ is locally convex-solid iff there exists a family 
$\mathcal{M}=\{m_\lambda\}_{\lambda\in\Lambda}$ of lattice seminorms that generates $\tau$ (cf. \cite[Thm.2.25]{Alip03}). Moreover, for such $\mathcal{M}$,
$x_\alpha\tc x$ iff $m_\lambda(x_\alpha-x)\xrightarrow[\alpha]{}0$ in $\mathbb{R}$ for each $m_\lambda\in\mathcal{M}$. Since $\tau$ is Hausdorff then the family 
$\mathcal{M}$ is separating.
	
A subset $A$ in a topological vector space $(X,\tau)$ is called $\tau$-{\em bounded} if, for every $\tau$-neighborhood $V$ of zero, 
there exists $\lambda>0$ such that $A\subseteq\lambda V$. In the case when the topology $\tau$ is generated by a family $\{m_\lambda \}_{\lambda\in\Lambda}$ of seminorms,
a subset $A$ of $X$ is $\tau$-bounded iff $\sup_{a\in A}m_\lambda(a)<\infty$ for all $\lambda\in\Lambda$.
	
Recall that a locally solid vector lattice $(X,\tau)$ is said to have the {\em Lebesgue property} if $x_\alpha\downarrow 0$ in $X$ implies $x_\alpha\tc 0$ or, 
equivalently, if $x_\alpha\oc 0$ implies $x_\alpha\tc 0$; $(X,\tau)$ is said to have the {\em $\sigma$-Lebesgue property} if $x_n\downarrow 0$ in $X$ implies $x_n\tc 0$; 
and $\left( X,\tau\right)$ is said to have the {\em pre-Lebesgue} property if $0\leq x_n\uparrow\ \leq x$ implies only that $x_n$ is $\tau$-Cauchy. 
Finally, $\left(X,\tau\right)$ is said to have the {\em Levi property} if, when $0\leq x_\alpha\uparrow$ and $x_\alpha$ is $\tau$-bounded, 
then $x_\alpha\uparrow x$ for some $x\in X$; $(X,\tau)$ is said to have the {\em $\sigma$-Levi property} if $x_n$ has supremum in $X$ provided
by $0\leq x_n\uparrow$ and by the $\tau$-boundedness of $x_n$, see \cite[Def. 3.16]{Alip03}.

\section{Multi-Normed Vector Lattices}

Let $(X,\tau)$ be a locally convex-solid vector lattice with an upward directed family $\mathcal{M}=\{m_\lambda\}_{\lambda\in\Lambda}$ of lattice seminorms generating $\tau$. 
Throughout this article, the pair $(X,\mathcal{M})$ will be referred to as a {\em multi-normed vector lattice $($MNVL$)$}. Also, $\tau$-convergence, $\tau$-Cauchy, $\tau$-complete, etc. 
will be denoted by $m$-convergence, $m$-Cauchy, $m$-complete, etc.

Let $X$ be a vector space, $E$ be a vector lattice, and $p:X\to E_+$ be a vector norm (i.e.\  $\,p(x)=0\Leftrightarrow x=0$, 
$p(\lambda x)=|\lambda|p(x)$ for all $\lambda\in\mathbb{R}$, $x\in X$, and $p(x+y)\leq p(x)+p(y)$ for all $x,y\in X$), 
then $(X,p,E)$ is called a {\em lattice-normed space}, abbreviated as {\em LNS}, see \cite{K}. If $X$ is a vector lattice, 
and the vector norm $p$ is monotone (i.e.\ $\abs{x}\leq\abs{y}\Rightarrow p(x)\leq p(y)$), then the 
triple $(X,p,E)$ is called a {\em {\em lattice-normed vector lattice}}, abbreviated as {\em LNVL} (cf. \cite{AEEM1,AEEM2}).

Given an LNS $(X,p,E)$. Recall that a net $x_\alpha$ in $X$ is said to be {\em p-convergent} to $x$ (see \cite{AEEM1}) if $p(x_\alpha-x)\oc 0$ in $E$. In this case, we write $x_\alpha\pc x$. 
A subset $A$ of $X$ is called {\em p-bounded} if there exists $e\in E$ such that $p(a)\leq e$ for all $a\in A$.
	
\begin{prop}\label{MNVL is LNVL}
Every MNVL induces an LNVL. Moreover, for arbitrary nets, $p$-convergence in the induced LNVL implies $m$-convergence, and they coincide in the case of $p$-bounded nets.
\end{prop}

\begin{proof}
Let $(X,\mathcal{M})$ be an MNVL, then there is a separating family $\{m_{\lambda}\}_{\lambda\in\Lambda}$ of lattice seminorms on $X$. Let $E=\mathbb{R}^{\Lambda}$ be 
the vector lattice of all real-valued functions on $\Lambda$, and define $p:x\mapsto p_x$ from $X$ into $E_+$ such that $p_x[\lambda]\coloneqq m_{\lambda}(x)$.
		
It is clear that $p$ is a monotone vector norm on $X$. Therefore $(X,p,E)$ is an LNVL. Let $x_\alpha$ be a net in $X$. If $x_\alpha\pc 0$, then ${p_x}_\alpha\oc 0$ in $\mathbb{R}^\Lambda$, and so ${p_x}_\alpha [\lambda]\to 0$ or $m_\lambda(x_\alpha)\to 0$ 
for all $\lambda\in\Lambda$. Hence $x_\alpha\mc 0$.
		
Finally, assume a net $x_\alpha$ to be $p$-bounded. If $x_\alpha\mc 0$, then $m_\lambda(x_\alpha)\to 0$ or  ${p_x}_\alpha [\lambda]\to 0$ for each $\lambda\in\Lambda$. 
Since $x_\alpha$ is $p$-bounded, then ${p_x}_\alpha\oc 0$ in $\mathbb{R}^\Lambda$. That is $x_\alpha\pc 0$.
\end{proof}

Let $X$ be a vector lattice. An element $0\neq e\in X_+$ is called a {\em strong unit} if the ideal $I_e$ generated by $e$ is $X$ or, equivalently, for every $x\ge 0$, 
there exists $n\in\mathbb N$ such that $x\le ne$; a {\em weak unit} if the band $B_e$ generated by $e$ is $X$ or, equivalently, $x\wedge ne\uparrow x$ for every $x\in X_+$. 
If $(X,\tau)$ is a topological vector lattice, then $0\neq e\in X_+$ is called a {\em quasi-interior point} if the principal ideal $I_e$ is $\tau$-dense in $X$
(see Definition 6.1 in \cite{Sch74}). It is known that
\begin{displaymath}
	\text{strong unit}\Rightarrow
	\text{quasi-interior point}\Rightarrow
	\text{weak unit}.
\end{displaymath}
The following proposition characterizes quasi-interior points, and should be compared with \cite[Thm.4.85]{Aliprantis:06}.
	
\begin{prop}\label{um-convergence at m-q.i.p}
Let $(X,\mathcal{M})$ be an MNVL, then the following statements are equivalent$:$
\begin{enumerate}
		\item\label{tqip} $e\in X_+$ is a quasi-interior point;
		\item\label{mqip} for all $x\in X_+$, $x-x\wedge ne\mc 0$ as $n\rightarrow\infty$;
		\item\label{strictly + in *} $e$ is strictly positive on $X^*$, i.e., $0<f\in X^*$ implies $f( e) > 0$, where $X^*$ denotes the topological dual of $X$.
\end{enumerate}
\end{prop}

\begin{proof}
\eqref{tqip}$\Rightarrow$\eqref{mqip} 
Suppose that $e$ is a quasi-interior point of $X$, then ${\overline{I_e}}^m = X$. Let $x\in X_+$. 
Then $x\in{\overline{I_e}}^m$, so there exists a net $x_\alpha$ in $I_e$ that $m$-converges to $x$. But $x_\alpha\mc x$ implies $\abs{x_\alpha}\mc\abs{x}= x$. 
Moreover, $x_\alpha\wedge x\mc x\wedge x =x$, and $x_\alpha\wedge x\leq x_\alpha$ implies that  $x_\alpha\wedge x\in I$, because $I_e$ is an ideal. 
So we can assume also that $x_\alpha\leq x$. Hence, for any $x\in X_+$, there is a net $0\leq x_\alpha\in I_e$ and $x_\alpha\leq x$. 
Then $0\leq x_\alpha\wedge ne\leq x\wedge ne \leq x$ for all $n\in\mathbb{N}$. Now, take $\lambda\in\Lambda$, and let $\varepsilon>0$, 
then there is $\alpha_\varepsilon$ such that $m_\lambda(x-x_{\alpha_\varepsilon})<\varepsilon$. But $0\leq x_{\alpha_\varepsilon}\in I_e$, 
so $0\leq x_{\alpha_\varepsilon}\leq k_\varepsilon e$ for some $k_\varepsilon\in\mathbb{N}$. 
Since $0\leq x_{\alpha_\varepsilon}=x_{\alpha_\varepsilon}\wedge k_\varepsilon e\leq x\wedge k_\varepsilon e\leq x$, 
then $m_\lambda(x-x\wedge ne)\leq m_\lambda(x-x\wedge k_\varepsilon e)\leq m_\lambda(x-x_\alpha\wedge k_\varepsilon e)=m_\lambda(x-x_{\alpha_\varepsilon})<\varepsilon$ 
for all $n\geq k_\varepsilon$. Hence $m_\lambda(x-x\wedge ne)\to 0$ as $n\to\infty$. Since $\lambda\in\Lambda$ was chosen arbitrary, we get $x-x\wedge ne\mc 0$.

The proofs of the implications \eqref{mqip}$\Rightarrow$\eqref{strictly + in *}, and \eqref{strictly + in *}$\Rightarrow$\eqref{tqip} are similar to the proofs of the corresponding implications of Theorem 4.85 in \cite{Aliprantis:06}. 

\end{proof}

\section{$um$-Topology}

In this section we introduce the $um$-topology in a analogous manner to the $un$-topology \cite{KMT} and $uaw$-topology \cite{Zab}. First we define the $um$-convergence.
\begin{definition}
Let $(X,\mathcal{M})$ be an MNVL, then a net $x_{\alpha}$ is said to be {\em unbounded m-convergent to $x$}, if $\abs{x_\alpha-x}\wedge u\mc 0$ for all $u\in X_{+}$. 
In this case, we say $x_\alpha$ {\em $um$-converges} to $x$ and write $x_{\alpha}\umc x$.
\end{definition}

Clearly, that $um$-convergence is a generalization of $un$-convergence. The following result generalizes \cite[Cor.4.5]{KMT}.

\begin{prop}
If $(X,\mathcal{M})$ is an MNVL possessing the Lebesgue and Levi properties, and $x_\alpha\umc 0$ in $X$, then $x_\alpha\umc 0$ in $X^{**}$.
\end{prop}

\begin{proof}
It follows from Theorem 6.63 of \cite{Alip03} that  $(X,\mathcal{M})$ is $m$-complete and $X$ is a band in $X^{**}$. 
Now, \cite[Thm.2.22]{Alip03} shows that $X^{**}$ is Dedekind complete, and so $X$ is a projection band in $X^{**}$. 
The conclusion follows now from \cite[Thm.3(3)]{DEM2}.
\end{proof}

In a similar way as in \cite[Section 7]{DOT}, one can show that $\mathcal{N}_0$, the collection of all sets of the form 
$$
  V_{\varepsilon,u,\lambda}=\{x\in X: m_{\lambda}(\abs{x}\wedge u)<\varepsilon\} \ \ \ \ (\varepsilon>0, 0\neq u\in X_+, \lambda\in\Lambda) 
$$ 
forms a neighborhood base at zero for some Hausdorff locally solid topology $\tau$ such that, for any net $x_\alpha$ in $X$: $x_\alpha\umc 0$ iff $x_\alpha\tc 0$. 
Thus, the $um$-convergence is topological, and we will refer to its topology as the {\em um-topology}. 

Clearly, if $x_\alpha\mc 0$, then $x_\alpha\umc 0$, and so the $m$-topology, in general, is finer than $um$-topology. 
On the contrary to Theorem 2.3 in \cite{KMT}, the following example provides an MNVL which has a strong unit, yet the $m$-topology and $um$-topology do not agree.
	
\begin{exam}
		Let $X=C[0,1]$. Let $\mathcal{A}:=\{[a,b]\subseteq [0,1]: a<b\}$. For $[a,b]\in\mathcal{A}$ and $f\in X$, let $m_{[a,b]}(f):=\frac{1}{b-a}\int _{a}^{b}\lvert f(t)\rvert dt$. 
		Then $\mathcal{M}=\{m_{[a,b]}:[a,b]\in\mathcal{A}\}$ is a separating family of lattice seminorms on $X$. Thus, $(X,\mathcal{M})$ is an MNVL. 
		For each $2\leq n \in\mathbb{N}$, let 
		\[
		f_n=
\begin{cases}
		n &\text{if } x\in {[0,\frac{1}{n}]},\\
		n^2( 1-n ) x + n^2 &\text{if } x\in {[\frac{1}{n},\frac{1}{n-1}]},\\
		0 &\text{if } x\in{[\frac{1}{n-1},1]}.
\end{cases}
		\]
		So we have
		\[
		f_n\wedge\one=
		\begin{cases}
		1 &\text{if } x\in{[0,\frac{n+1}{n^2}]},\\
		n^2( 1-n ) x + n^2 &\text{if } x\in{[\frac{n+1}{n^2},\frac{1}{n-1}]},\\
		0 &\text{if } x\in{[\frac{1}{n-1},1]}.
		\end{cases}
		\]
		Now, let $0<b\leq 1$, then there is $n_0\in\mathbb{N}$ such that $\frac{1}{n_0 -1 } < b$. So, for $n\geq n_0$, we have $\frac{1}{n-1} < b$, and so 
		we get $m_{[0,b]}(f_n)=\frac{1}{b}(1+\frac{1}{n-1})\to\frac{1}{b}\neq 0$ as $n\to\infty$. Thus, $f_n\not\mc 0$. On the other hand, if $[a,b]\in\mathcal{A}$ 
		then there is $n_0\in\mathbb{N}$ such that $\frac{1}{n_0-1}<b$ so, for $n\geq(n_0-1)$, we have $m_{[a,b]}(f_n\wedge\one)=\frac{1}{b-a}(\frac{n+1}{n^2}+\frac{1}{2n^2(n-1)})\to 0$ 
		as $n\to\infty$. Since $\one$ is a strong unit in $X$ then, by \cite[Cor.5]{DEM2}, $f_n\umc 0$. 
\end{exam}

\section{Metrizabililty of $um$-topology}

The main result in this section is Proposition \ref{metrizable iff countable t.o.s}, which shows that the $um$-topology is metrizable iff the space has a countable topological orthogonal system.

It is well known (cf. \cite[Thm.2.1]{Alip03}) that a topological 
vector space is metrizable iff it has a countable neighborhood base at zero. Furthermore, an MNVL $(X,\mathcal{M})$ is metrizable iff the $m$-topology is generated by a countable family of lattice seminorms, see \cite[Theorem VII.8.2]{Vulikh67}.

Notice that, in an MNVL $(X,\mathcal{M})$ with countable $\mathcal{M}=\{m_k\}_{k\in\mathbb{N}}$, an equivalent translation-invariant metric $\rho_{\mathcal{M}}$ 
can be constructed by the formula
\begin{equation}\label{metric formula}
  \rho_{\mathcal{M}}(x,y)=\sum\limits_{k=1}^{\infty}\frac{m_k(x-y)}{2^k(m_k(x-y)+1)} \ \ \ (x,y\in X).
\end{equation}
Since the function $t\to\frac{t}{t+1}$ is increasing on $[0,\infty)$, $|x|\le|y|$ in $X$ implies that $\rho_{\mathcal{M}}(x,0)\le\rho_{\mathcal{M}}(y,0)$.\\
	
Recall that a collection $\{e_\gamma\}_{\gamma\in\Gamma}$ of positive vectors in a vector lattice $X$ is called an {\em orthogonal system} if $e_\gamma\wedge e_{\gamma'}=0$ 
for all $\gamma\neq\gamma'$. If, moreover, $x\wedge e_\gamma=0$ for all $\gamma\in\Gamma$ implies $x=0$, then $\{e_\gamma\}_{\gamma\in\Gamma}$ is called a {\em maximal orthogonal system}.
It follows from the Zorn's lemma that every vector lattice containing at least one non-zero element has a maximal orthogonal system. 
Next, we recall the following notion.

\begin{definition}\cite[Def.1]{DEM2}
Let $(X,\tau)$ be a topological vector lattice. An orthogonal system $Q=\{e_\gamma\}_{\gamma\in\Gamma}$ of non-zero elements in $X_+$ is said 
to be a {\em topological orthogonal system}, if the ideal $I_Q$ generated by $Q$ is $\tau$-dense in $X$.
\end{definition}

A series $\sum_{i=1}^{\infty} x_i$ in a multi-normed space $(X,\mathcal{M})$ is called {\em absolutely $m$-convergent} if $\sum_{i=1}^{\infty}m_\lambda(x_i)<\infty$ for all $\lambda\in\Lambda$; 
and the series is {\em $m$-convergent}, if the sequence $s_n\coloneqq\sum_{i=1}^{n}x_i$ of partial sums is $m$-convergent. The following lemma can be proven by combining a diagonal argument with the proof of \cite[Prop. 3 in Section 3.3]{Jarchow} and therefore we omit its proof.

\begin{lem}\label{comp. iff -abs conv iff conv}
	A metrizable multi-normed space $(X,\mathcal{M})$ is $m$-complete iff every absolutely $m$-convergent series in $X$ is $m$-convergent. 
\end{lem}

The following result extends \cite[Thm.3.2]{KMT}.

\begin{prop}\label{metrizable iff countable t.o.s}
Let $(X,\mathcal{M})$ be a metrizable $m$-complete MNVL. Then the following conditions are equivalent$:$

$(i)$   $X$ has a countable topological orthogonal system$;$

$(ii)$  the $um$-topology is metrizable$;$ 

$(iii)$ $X$ has a quasi interior point. 
\end{prop}

\begin{proof}
Since $(X,\mathcal{M})$ is metrizable, we may suppose that $\mathcal{M}=\{m_k\}_{k\in\mathbb{N}}$ is countable and directed.

$(i) \Rightarrow (ii)$ 
It follows directly from \cite[Prop.5]{DEM2}. Notice also that a metric $d_{um}$ of the $um$-topology can be constructed by the following formula:
\begin{equation}\label{um-metric}
	d(x,y)=\sum_{k,n=1}^\infty \frac{1}{2^{k+n}}\cdot\frac{m_k(|x-y|\wedge e_n)}{1+m_k(|x-y|\wedge e_n)},
\end{equation}
where $\{e_n\}_{n\in\mathbb{N}}$ is a countable topological orthogonal system for $X$. 
		
$(ii)\Rightarrow(iii)$ 
Assume that the $um$-topology is generated by a metric $d_{um}$ on $X$. For each $n\in\mathbb{N}$, 
let $B_{um}(0,\frac{1}{n})=\lbrace x\in X:d_{um}(x,0)<\frac{1}{n}\rbrace$. Since the $um$-topology is metrizable, then, for each $n\in\mathbb{N}$, 
there are $k_n\in\mathbb{N}, 0<u_n\in X_+$, and $\varepsilon_n>0$ such that $V_{\varepsilon_n,u_n,k_n}\subseteq B_{um}(0,\frac{1}{n})$, where 
$$
  V_{\varepsilon,u_n,k}=\{x\in X: m_{k}(|x|\wedge u_n)<\varepsilon\}.
$$ 
Notice that $\{V_{\varepsilon,u_n,k}\}_{\varepsilon>0,n,k\in\mathbb{N}}$ is a base at zero of the $um$-topology on $X$.

Let $B_m(0,1)=\lbrace x\in X:d_m(x,0)<1\rbrace$, 
where $d_m$ is the metric generating the $m$-topology. There is a zero neighborhood $V$ in the $m$-topology such that $V\subseteq B_m(0,1)$. 
Since $V$ is absorbing, then, for every $n\in\mathbb{N}$, there is $c_n\geq 1$ such that $\frac{1}{c_n}u_n\in V$. 
Thus $\frac{1}{c_n}u_n\in V\subseteq B_m(0,1)$ for each $n\in\mathbb{N}$. 
Hence, the sequence $\frac{1}{c_n}u_n$ is $d_m$-bounded and so it is bounded with respect to the multi-norm $\mathcal{M}=\{m_k\}_{k\in\mathbb{N}}$. Let 
\begin{equation}\label{e prop7}
		e\coloneqq\sum_{n=1}^\infty\dfrac{u_n}{2^n c_n}.
\end{equation}
Fix $k\in\mathbb{N}$. Since the sequence $\frac{u_n}{c_n}$ is bounded with respect to $\mathcal{M}$, there exists $r_k\in\mathbb{R}_+$ 
such that $m_k(\frac{u_n}{c_n})\leq r_k<\infty$ for all $n\in\mathbb{N}$. Hence, 
\begin{align*}
		\sum_{n=1}^\infty m_k\bigg(\dfrac{u_n}{2^n c_n}\bigg)=\sum_{n=1}^{\infty}\dfrac{1}{2^n}m_k\bigg(\frac{u_n}{c_n}\bigg) 
		&\leq r_k\sum_{n=1}^\infty\dfrac{1}{2^n}<\infty.
\end{align*}
Thus, the series $\sum_{n=1}^\infty\dfrac{u_n}{2^n c_n}$ is absolutely $m$-convergent. 
Since $X$ is $m$-complete, Lemma \ref{comp. iff -abs conv iff conv} assures that the series $\sum_{n=1}^\infty\dfrac{u_n}{2^n c_n}$ is $m$-convergent to some $e\in X$.
		
Now, we use Theorem 2 in \cite{DEM2} to show that $e$ is a quasi-interior point in $X$. Let $x_\alpha$ be a net in $X_+$ such that $x_\alpha\wedge e\mc0$. 
Our aim is to show that $x_\alpha\umc 0$. Since 
$$
  x_\alpha\wedge u_n\leq2^n c_n x_\alpha\wedge2^n c_n e=2^nc_n(x_\alpha\wedge e)\mc 0 \ \ \ \ (\alpha\to\infty),
$$ 
then $x_\alpha\wedge u_n\mc0$ for all $n\in\mathbb{N}$. In particular, $m_{k_n}(x_\alpha\wedge u_n)\to0$. Thus, there exists $\alpha_n$ such that  
$m_{k_n}(x_\alpha\wedge u_n)<\varepsilon_n$ for all $\alpha\geq\alpha_n$. That is $x_\alpha\in V_{\varepsilon_n,u_n,k_n}$ for all $\alpha\geq\alpha_n$, 
which implies $x_\alpha\in B_{um}(0,\frac{1}{n})$. Therefore, $x_\alpha\xrightarrow{\mathrm{d_{um}}}0$ and so $x_\alpha\umc0$. 
Hence, $e$ is a quasi interior point.

$(iii)\Rightarrow(i)$ It is trivial. 
\end{proof}

Similar to \cite[Prop.3.3]{KMT}, we have the following result.
	
\begin{prop}
Let $(X,\mathcal{M})$ be an $m$-complete metrizable MNVL. The $um$-topology is stronger than a metric topology iff $X$ has a weak unit.
\end{prop}

\begin{proof}
The sufficiency follows from \cite[Prop.6]{DEM2}.

For the necessity, suppose that the $um$-topology is stronger than the topology generated by a metric $d$. 
Let $e$ be as in (\ref{e prop7}) above. Assume $x\wedge e=0$. Since $e\geq\frac{u_n}{2^n c_n}$ for all $n\in\mathbb{N}$, 
we get $x\wedge\frac{u_n}{2^n c_n}=0$, and hence $x\wedge u_n=0$ for all $n$. Then $x\in V_{\varepsilon_n,u_n,k_n}$ for all $n$,
and $x\in B(0,\frac{1}{n})=\lbrace x\in X: d(x,0)<\frac{1}{n}\rbrace$ for each $n\in\mathbb{N}$. So $x=0$, which means that 
$e$ is a weak unit.
\end{proof}

\section{$um$-Completeness}

A subset $A$ of an MNVL $(X,\mathcal{M})$ is said to be {\em $($sequentially$)$ $um$-complete} if, it is (sequentially) complete in the $um$-topology.
In this section, we characterize $um$-complete subsets of $X$ in terms of the Lebesgue and Levi properties. We begin with the following technical lemma.  

\begin{lem}\label{um-closure is m-bdd}
Let  $(X,\mathcal{M})$ be an MNVL, and $A\subseteq X$ be $m$-bounded, then $\overline{A}^{um}$ is $m$-bounded.
\end{lem}

\begin{proof}
Given $\lambda\in\Lambda$, then $M_{\lambda}=\sup_{a\in A} m_\lambda(a)<\infty$. Let $x\in\overline{A}^{um}$, then there is a net $a_\alpha$ in $A$ such that $a_\alpha\umc x$. 
So $m_\lambda(\abs{a_\alpha-x}\wedge u)\rightarrow 0$ for any $u\in X_+$. In particular,
\begin{align*}
		m_\lambda(\vert x\vert)&=m_\lambda(\vert x\vert\wedge\vert x\vert)=m_\lambda(\vert x-a_\alpha+a_\alpha\vert\wedge\vert x\vert)\leq\\
		& m_\lambda(\vert x-a_\alpha\vert\wedge\vert x\vert)+\sup_{a\in A}m_\lambda(a)=m_\lambda(\vert x-a_\alpha\vert\wedge\vert x\vert)+M_{\lambda}.
\end{align*}
Letting $\alpha\rightarrow\infty$, we get $m_\lambda(x)=m_\lambda(\vert x\vert)\leq M_{\lambda}<\infty$ for all $x\in\overline{A}^{um}$.
\end{proof}
	
\begin{thm}\label{um-completeness}
Let $(X,\mathcal{M})$ be an MNVL and let $A$ be an $m$-bounded and $um$-closed subset in $X$. If $X$ has the Lebesgue and Levi properties, then $A$ is $um$-complete.
\end{thm}

\begin{proof}
Suppose that $x_\alpha$ is $um$-Cauchy in $A$, then, without lost of generality, we may assume that $x_\alpha$ consists of positive elements.\\
Case (1): 
If $X$ has a weak unit $e$, then $e$ is a quasi-interior point, by the Lebesgue property of $X$ and Proposition \ref{um-convergence at m-q.i.p}. 
Note that, for each $k\in\mathbb{N}$,
\begin{equation*}
		\abs{ x_\alpha\wedge ke-x_{\beta}\wedge ke}\leq\abs{ x_\alpha-x_\beta}\wedge ke,
\end{equation*}
hence the net $(x_\alpha\wedge ke)_\alpha$ is $m$-Cauchy in $X$. Now, \cite[Thm.6.63]{Alip03} assures that $X$ is $m$-complete, 
and so the net $(x_\alpha\wedge ke)_\alpha$ is $m$-convergent to some $y_k\in X$. Given $\lambda\in\Lambda$. Then
\begin{align*}
		m_\lambda(y_k)&=m_\lambda( y_k-x_\alpha\wedge ke+x_\alpha\wedge ke) \\
		&\leq m_\lambda( y_k-x_\alpha\wedge ke) +m_\lambda( x_\alpha) \\
		&\leq m_\lambda( y_k-x_\alpha\wedge ke)+\sup_\alpha m_\lambda(x_\alpha).
\end{align*}
Taking limit over $\alpha$, we get $m_\lambda(y_k)\leq\sup_\alpha m_\lambda(x_\alpha)<\infty$. Hence the sequence $y_k$ is $m$-bounded in $X$.	
Note also that  $y_k$ is increasing in $X$, but $X$ has the Lebesgue and Levi properties, so, by \cite[Thm.6.63]{Alip03}, $y_k$ $m$-converges to some $y\in X$. \\
		
It remains to show that $y$ is the $um$-limit of $x_\alpha$. Given $\lambda\in\Lambda$. Note that, by Birkhoff's inequality,
\begin{align*}
		\abs{x_\alpha\wedge ke-x_\beta\wedge ke}\wedge e &\leq\abs{x_\alpha-x_\beta}\wedge e.
\end{align*}
Thus 
\begin{equation*}
		m_\lambda(\abs{x_\alpha\wedge ke-x_\beta\wedge ke}\wedge e )\leq m_\lambda (\abs{x_\alpha-x_\beta}\wedge e).
\end{equation*}
Taking limit over $\beta$, we get 
\begin{equation*}
		m_\lambda(\abs{x_\alpha\wedge ke-y_k}\wedge e )\leq\lim_\beta m_\lambda (\abs{x_\alpha-x_\beta}\wedge e).
\end{equation*}
Now taking limit over $k$, we have 
\begin{equation*}
		m_\lambda(\abs{x_\alpha-y}\wedge e )\leq\lim_\beta m_\lambda(\abs{x_\alpha-x_\beta}\wedge e).
\end{equation*}
Finally, as $x_\alpha$ is $um$-Cauchy, taking limit over $\alpha$, yields
\begin{equation*}
		\lim_{\alpha} m_\lambda(\abs{x_\alpha-y}\wedge e )\leq\lim_{\alpha,\beta} m_\lambda(\abs{x_\alpha-x_\beta}\wedge e)=0.
\end{equation*}
Thus, $x_\alpha\umc y$ and, since $A$ is $um$-closed, $y\in A$.\\
		
Case (2): 
If $X$ has no weak unit. Let $\{e_\gamma\}_{\gamma\in\Gamma}$ be a maximal orthogonal system  in $X$. Let $\Delta$ be the collection of all finite subsets of $\Gamma$. 
For each $\delta\in\Delta, ~\delta=\{\gamma_1,\gamma_2,\dots,\gamma_n\}$, consider the band $B_\delta$ generated by $\{  e_{\gamma_1}, e_{\gamma_1},\dots,e_{\gamma_n}\}$. 
It follows from \cite[Thm.3.24]{Alip03}  that  $B_\delta$ is a projection band. Then  $B_\delta$ is an $m$-complete MNVL in its own right. 
Moreover, the $m$-topology restricted to $B_\delta$ possesses the Lebesgue and Levi properties.	Note that $B_\delta$ has a weak unit, namely $e_{\gamma_1}+e_{\gamma_2}+\dots+e_{\gamma_n}$. 
Let $P_\delta$ be the band projection corresponding to $B_\delta$.\\
For $\delta\in\Delta$, since $x_\alpha$ is $um$-Cauchy in $X$ and $P_\delta$ is a band projection, then $P_\delta x_\alpha$ is $um$-Cauchy in $B_\delta$. 
Lemma \ref{um-closure is m-bdd} assures that $\overline{P_\delta (A)}^{um}$ is $m$-bounded in $B_\delta$. Thus, by Case (1), there is $z_\delta\in B_\delta$ such that 
\begin{equation*}
		P_\delta x_\alpha\umc z_\delta\geq 0~\text{in~} B_\delta \ \ \ \ (\alpha\to\infty).
\end{equation*}
Since $B_\delta$ is a projection band, then $P_\delta x_\alpha\umc z_\delta\geq 0~\text{in~} X~~~(\text{over~}\alpha)$. 
It is easy to see that $0\leq z_\delta\uparrow$, and $z_\delta$ is $m$-bounded. Since $X$ has the Lebesgue and Levi properties, it follows from \cite[Thm.6.63]{Alip03}, 
that there is $z\in X_+$ such that $z_\delta\mc z$, and so $z_\delta\uparrow z$. It remains to show that $x_\alpha\umc z$. The argument is similar to the proof of 
\cite[Thm.4.7]{GaoX:14}, and we leave it as an exercise. Since $A$ is $um$-closed, then $z\in A$ and so $A$ is $um$-complete.
\end{proof}
	
The following lemma and its proof are analogous to Lemma 1.2 in \cite{KMT}.

\begin{lem}\label{KMT 2 Lemmas} 
Let $(X,\mathcal{M})$ be an MNVL. If $x_\alpha$ is an increasing net in $X$ and $x_\alpha\umc x$, then $x_\alpha\uparrow x$ and $x_\alpha\mc x$.
\end{lem}

\begin{lem}\label{Cauchy but not um-convergent}
Let $(X,\mathcal{M})$ be an MNVL possessing the pre-Lebesgue property. Let $x_n$ be a positive disjoint sequence which is not $m$-null. 
Put $s_n\coloneqq\sum_{k=1}^n x_k$. Then the sequence $s_n$ is $um$-Cauchy, which is not $um$-convergent.
\end{lem}

\begin{proof}
The sequence $s_n$ is monotone increasing and, since $x_n$ is not $m$-null, $s_n$ is not $m$-convergent. 
Hence, by Lemma \ref{KMT 2 Lemmas}, the sequence $s_n$ is not $um$-convergent. To show that $s_n$ is $um$-Cauchy, fix any $\varepsilon > 0$ and take $0\neq w\in X_+$. 
Since $x_n$ is a positive disjoint sequence, we have $s_n\wedge w =\sum_{k=1}^n w \wedge x_k$. 
The sequence $s_n\wedge w$ is increasing and order bounded by $w$, hence it is $m$-Cauchy, by \cite[Thm.3.22]{Alip03}. 
Let $\lambda\in\Lambda$. We  can find $n_{\varepsilon_\lambda}$ such that $m_\lambda(s_m\wedge w -s_n\wedge w)<\varepsilon$ for all $m\geq n\geq n_{\varepsilon_\lambda}$. 
Observe that 
\begin{align*}
		s_m\wedge w - s_n\wedge w &=\sum_{k=1}^{m} w \wedge x_k -\sum_{k=1}^{n} w \wedge x_k\\
		&=\sum_{k=n+1}^{m}  w \wedge x_k =w \wedge  \sum_{k=n+1}^{m} x_k =w\wedge \abs{s_m - s_n}.
\end{align*}
It follows $m_\lambda(\abs{s_m - s_n}\wedge w)<\varepsilon$ for all  $m\geq n\geq n_{\varepsilon_\lambda}$. But $\lambda\in\Lambda$ was chosen arbitrary. Hence $s_n$ is $um$-Cauchy.
\end{proof}

Next theorem generalizes Theorem 6.4 in \cite{KMT}.

\begin{thm}\label{m-bounded um-closed implies X is um-complete}
Let $(X,\mathcal{M})$ be an $m$-complete MNVL with the pre-Lebesgue property. Then $X$ has the Lebesgue and Levi properties iff every $m$-bounded $um$-closed subset of $X$ is $um$-complete.
\end{thm}

\begin{proof}
The necessity follows directly from Theorem \ref{um-completeness}.

For the sufficiency, first notice that, in an $m$-complete MNVL, the pre-Lebesgue and Lebesgue properties coincide \cite[Thm.3.24]{Alip03}. 

If $X$ does not have the Levi property then, by \cite[Thm.6.63]{Alip03}, there is a disjoint sequence $x_n\in X_+$, which is not $m$-null, 
such that its sequence of partial sums $s_n=\sum_{j=1}^n x_j$ is $m$-bounded.
Let $A=\overline{\{s_n:n\in\mathbb{N}\}}^{um}$. By Lemma \ref{um-closure is m-bdd}, we have that $A$ is $m$-bounded. 
By Lemma \ref{Cauchy but not um-convergent}, the sequence $s_n$ is $um$-Cauchy in $X$ and so in $A$, in contrary 
with that the sequence $s_{n+1}-s_n=x_{n+1}$ is not $m$-null. 
\end{proof}  
	
\begin{thm}\label{seq. um-completeness}
Let $(X,\mathcal{M})$ be an $m$-complete metrizable MNVL, and let $A$ be an $m$-bounded sequentially $um$-closed subset of $X$. 
If $X$ has the $\sigma$-Lebesgue and $\sigma$-Levi properties then $A$ is sequentially $um$-complete. Moreover, the converse holds if, in addition, $X$ is Dedekind complete.
\end{thm}	

\begin{proof}
Suppose $\mathcal{M}=\{m_k\}_{k\in\mathbb{N}}$. Let $0\leq x_n$ be a $um$-Cauchy sequence in $A$. Let $e\coloneqq\sum_{n=1}^\infty\frac{x_n}{2^n}$.	For $k\in\mathbb{N}$, 
$$
	\sum_{n=1}^\infty m_k\bigg(\frac{x_n}{2^n}\bigg)=\sum_{n=1}^\infty\frac{1}{2^n} m_k (x_n)\leq c_k\sum_{n=1}^\infty\frac{1}{2^n}<\infty,
$$ 
where $m_k(a)\leq c_k <\infty$ for all $a\in A$. Since $\sum_{n=1}^{\infty}\frac{x_n}{2^n}$ is absolutely $m$-convergent, then, by Lemma \ref{comp. iff -abs conv iff conv}, 
$\sum_{n=1}^\infty\frac{x_n}{2^n}$ is $m$-convergent in $X$. Note that, $x_n\leq 2^ne$, so $x_n\in B_e$ for all $n\in\mathbb{N}$. Since $X$ has the Levi property, 
then $X$ is $\sigma$-order complete (see \cite [Definition 3.16]{Alip03}). Thus $B_e$ is a projection band. Also $e$ is a weak unit in $B_e$. 
Then, by the same argument as in Theorem \ref{um-completeness}, we get that there is $x\in B_e$ such that $x_n\umc x$ in $B_e$ and so $x_n\umc x$ in $X$. 
Since $A$ is sequentially $um$-closed, we get $x\in A$. Thus $A$ is sequentially $um$-complete.
		
The converse follows from Proposition 8 in \cite{DEM2}.
\end{proof}

\section{$um$-Compact sets}
A subset $A$ of an MNVL $(X,\mathcal{M})$ is said to be {\em $($sequentially$)$ $um$-compact} if, it is (sequentially) compact in the $um$-topology. In this section, we characterize $um$-compact subsets of $X$ in terms of the Lebesgue and Levi properties. We begin with the following result which shows that $um$-compactness can be ``localized'' under certain conditions.
	
\begin{thm}\label{Theorem x um Px um} 
Let $(X,\mathcal{M})$ be an MNVL possessing the Lebesgue property. Let $\{e_\gamma\}_{\gamma\in\Gamma}$ be a maximal orthogonal system. 
For each $\gamma\in\Gamma$, let $B_\gamma$ be the band generated by $e_\gamma$, and $P_\gamma$ be the corresponding band projection onto $B_\gamma$. 
Then $x_\alpha\umc 0\text{~ in ~} X$ iff $P_\gamma x_\alpha\umc 0$ in $B_\gamma$ for all $\gamma\in\Gamma$. 
\end{thm}

\begin{proof}
For the forward implication, we assume that $x_\alpha\umc 0$ in $X$. Let $b\in (B_\gamma)_+$. Then 
$$
	\abs{P_\gamma x_\alpha}\wedge b=P_\gamma\abs{x_\alpha}\wedge b\leq\abs{x_\alpha}\wedge b\mc 0,
$$
that implies $P_\gamma x_\alpha\umc 0$ in $B_\gamma$.

For the backward implication, without lost of generality, we may assume that $x_\alpha\geq 0$ for all $\alpha$. Let $u\in X_+$. Our aim is to show that $x_\alpha\wedge u\mc 0$. 
It is known that $x_\alpha\wedge u=\sum_{\gamma\in\Gamma}P_\gamma(x_\alpha\wedge u)$. Let $F$ be a finite subset of $\Gamma$. Then
\begin{equation}\label{um 1}
		x_\alpha\wedge u=\sum_{\gamma\in F}P_\gamma(x_\alpha\wedge u)+\sum_{\gamma\in\Gamma\setminus F}P_\gamma(x_\alpha\wedge u).
\end{equation}
Note 
\begin{equation}\label{um 1`}
		\sum_{\gamma\in F}P_\gamma(x_\alpha\wedge u)=\sum_{\gamma\in F} P_\gamma x_\alpha \wedge P_\gamma u \mc 0.
\end{equation}
We have to control the second term in \eqref{um 1}.
\begin{equation}\label{um 2'}
		\sum_{\gamma\in\Gamma\setminus F}P_\gamma(x_\alpha\wedge u)\leq\frac{1}{n}\sum_{\gamma\in F} P_\gamma u + \sum_{\gamma\in\Gamma\setminus F} P_\gamma u,
\end{equation}
where $n\in\mathbb{N}$.
Let $\mathscr{F}(\Gamma)$ be the collection of all finite subsets of $\Gamma$.
Let $\Delta=\mathscr F(\Gamma)\times\mathbb{N}$. For each $\delta=(F,n)$, put
$$
  y_\delta=\frac{1}{n}\sum_{\gamma\in F} P_\gamma u + \sum_{\gamma\in\Gamma\setminus F} P_\gamma u.
$$
We show that $y_\delta$ is decreasing. Let $\delta_1\leq\delta_2$ then $\delta_1=(F_1,n_1),\delta_2=(F_2,n_2)$. Then $\delta_1\leq\delta_2$ iff $F_1\subseteq F_2$ and $n_1\leq n_2$. 
But $n_1\leq n_2$ iff $\frac{1}{n_1}\geq\frac{1}{n_2}$. So,
\begin{equation}\label{um 3}
		\frac{1}{n_1}\sum_{\gamma\in F_1} P_\gamma u\geq\frac{1}{n_2}\sum_{\gamma\in F_1} P_\gamma u.
\end{equation} 
Note also
\begin{equation}\label{um 5}
		\frac{1}{n_2}\sum_{\gamma\in F_2} P_\gamma u=\frac{1}{n_2}\sum_{\gamma\in F_1} P_\gamma u+\frac{1}{n_2}\sum_{\gamma\in F_2\setminus F_1} P_\gamma u.
\end{equation}
Since $F_1\subseteq F_2$, then $\Gamma\setminus F_1\supseteq\Gamma\setminus F_2$ and hence,
		$\sum_{\gamma\in\Gamma\setminus F_1} P_\gamma u\geq\sum_{\gamma\in\Gamma\setminus F_2} P_\gamma u$.
Note, that
\begin{equation}\label{um 7}
		\sum_{\gamma\in\Gamma\setminus F_1} P_\gamma u=\sum_{\gamma\in F_2\setminus F_1} P_\gamma u+\sum_{\gamma\in\Gamma\setminus F_2} P_\gamma u.
\end{equation}
Now,
\begin{equation}\label{um 8}
		\sum_{\gamma\in F_2\setminus F_1} P_\gamma u\geq\frac{1}{n_2}\sum_{\gamma\in F_2\setminus F_1} P_\gamma u.
\end{equation}
Combining \eqref{um 7} and \eqref{um 8}, we get 
\begin{equation}\label{um 9}
		\sum_{\gamma\in\Gamma\setminus F_1} P_\gamma u\geq\sum_{\gamma\in\Gamma\setminus F_2} P_\gamma u+\frac{1}{n_2}\sum_{\gamma\in F_2\setminus F_1} P_\gamma u.
\end{equation} 
Adding \eqref{um 3} and \eqref{um 9}, we get
\begin{equation*}
		\frac{1}{n_1}\sum_{\gamma\in F_1} P_\gamma u+\sum_{\gamma\in\Gamma\setminus F_1} P_\gamma u\geq\frac{1}{n_2}\sum_{\gamma\in F_1} 
   	P_\gamma u+\frac{1}{n_2}\sum_{\gamma\in F_2\setminus F_1} P_\gamma u+\sum_{\gamma\in\Gamma\setminus F_2} P_\gamma u.
\end{equation*}
It follows from \eqref{um 5}, that
\begin{equation*}
		\frac{1}{n_1}\sum_{\gamma\in F_1} P_\gamma u +\sum_{\gamma\in\Gamma\setminus F_1} P_\gamma u\geq\frac{1}{n_2}\sum_{\gamma\in F_2} P_\gamma u+\sum_{\gamma\in\Gamma\setminus F_2} P_\gamma u,
\end{equation*}
that is $y_{\delta_1}\geq y_{\delta_2}$. Next, we show $y_\delta\downarrow 0$. Assume $0\leq x\leq y_\delta$ for all $\delta\in\Delta$. Let $\gamma_0\in\Gamma$ be arbitrary and fix it. 
Let 
$$
  F=\{\gamma_0\}, \ n\in\mathbb{N}, \ 0\leq x\leq\frac{1}{n}P_{\gamma_0}u+\sum_{\gamma\in\Gamma\setminus\{\gamma_0\}}P_\gamma u.
$$ 
We apply $P_{\gamma_0}$ for the expression above, 
so $0\leq P_{\gamma_0}x\leq\frac{1}{n}P_{\gamma_0} u$ for all $n\in\mathbb{N}$, and so $P_{\gamma_0}x=0$. Since $\gamma_0\in\Gamma$ was chosen arbitrary, we get $P_{\gamma_0}x=0$ for all $\gamma\in\Gamma$. 
Hence, $x=0$ and so $y_\delta\downarrow 0$. Since $(X,\mathcal{M})$ has the Lebesgue property, we get $y_\delta\mc 0$. Therefore, by \eqref{um 2'}, 
\begin{equation}\label{um9}
		\sum_{\gamma\in\Gamma\setminus F} P_\gamma (x_\alpha\wedge u)\leq y_\delta\mc 0.
\end{equation}
Hence \eqref{um 1}, \eqref{um 1`}, and \eqref{um9} imply $x_\alpha\wedge u\mc 0$.
\end{proof}
	
The following result and its proof are similar to Theorem 7.1 in \cite{KMT}. Therefore we omit its proof.
	
\begin{thm}\label{Tychonoff-Like theorem}
Let $(X,\mathcal{M})$ be an MNVL possessing the Lebesgue and Levi properties. Let $\{e_\gamma\}_{\gamma\in\Gamma}$ be a maximal orthogonal system. Let $A$ be a $um$-closed $m$-bounded subset of $X$. 
Then $A$ is $um$-compact iff $P_\gamma(A)$ is $um$-compact in $B_\gamma$ for each $\gamma\in\Gamma$, where $B_\gamma$ is the band generated by $e_\gamma$ and $P_\gamma$ is 
the band projection corresponding to $B_\gamma$.
\end{thm}
	
\begin{thm}
Let $(X,\mathcal{M})$ be an $MNVL$. The following are equivalent$:$
\begin{enumerate}
		\item\label{umcompactA} Any $m$-bounded and $um$-closed subset $A$ of $X$ is $um$-compact.
		\item\label{umcompactAL} $X$ is an atomic vector lattice and $(X,\mathcal{M})$ has the Lebesgue and Levi properties.
\end{enumerate}
\end{thm}
	
\begin{proof}
\eqref{umcompactA} $\Rightarrow$ \eqref{umcompactAL}. Let $[a,b]$ be an order interval in $X$. For $x\in [a,b]$, we have $a\leq x\leq b$ and so $0\leq x-a\leq b-a$. 
Consider the order interval $[0,b-a]\subseteq X_+$. Clearly, $[0,b-a]$ is $m$-bounded and $um$-closed in $X$. By \eqref{umcompactA}, the order interval $[0,b-a]$ is $um$-compact. 
Let $x_\alpha$ be a net in $[0,b-a]$. Then there is a subset $x_{\alpha_\beta}$ such that $x_{\alpha_\beta}\umc x$ in $[0,b-a]$. That is $\abs{x_{\alpha_\beta}-x}\wedge u\mc 0$ for all $u\in [0,b-a]$. 
Hence, $\abs{x_{\alpha_\beta}-x}=\abs{x_{\alpha_\beta}-x}\wedge(b-a)\mc 0$. So, $x_{\alpha_\beta}\mc x$ in $[0,b-a]$. 
Thus, $[0,b-a]$ is $m$-compact. Consider the following shift operator $T_a:X\rightarrow X$ given by $T_a(x)\coloneqq x+a$. 
Clearly, $T_a$ is continuous, and so $T_a([0,b-a])=[a,b]$ is $m$-compact.\\
Since any order interval in $X$ is $m$-compact, then it follows from \cite[Cor.6.57]{Alip03} that $X$ is atomic and has the Lebesgue property. 
It remains to show that $X$ has the Levi property. 
Suppose $0\leq x_\alpha\uparrow$ and is $m$-bounded. Let $A=\overline{\{x_\alpha\}}^{um}$. Then $A$ is $um$-closed and, by Lemma \ref{um-closure is m-bdd}, $A$ is an $m$-bounded subset of $X$. 
Thus, $A$ is $um$-compact and so, there are a subnet $x_{\alpha_\beta}$ and $x\in A$ such that $x_{\alpha_\beta}\umc x$. Hence, by Lemma \ref{KMT 2 Lemmas}, $x_{\alpha_\beta}\uparrow x$, and so $x_\alpha\uparrow x$. 
Hence, $X$ has the Levi property.
		
\eqref{umcompactAL} $\Rightarrow$\eqref{umcompactA}. Let $A$ be an $m$-bounded and $um$-closed subset of $X$. We show that $A$ is $um$-compact.	Since $X$ is atomic, there is a maximal orthogonal system 
$\{e_\gamma\}_{\gamma\in\Gamma}$ of atoms. For each $\gamma\in\Gamma$, let $P_\gamma$ be the band projection corresponding to $e_\gamma$. Clearly, $P_\gamma(A)$ is $m$-bounded. 
Now, by the same argument as in the proof of Theorem 7.1 in \cite{KMT}, we get that $P_\gamma(A)$ is $um$-closed in $\prod_{\gamma\in\Gamma}B_\gamma$, and so it is $um$-closed in $B_\gamma$. 
But $um$-closedness implies $m$-closedness. So $P_\gamma(A)$ is $m$-bounded and $m$-closed in $B_\gamma$ for all $\gamma\in\Gamma$. Since each $e_\gamma$ is an atom in $X$, 
then $B_\gamma=\Span\{e_\gamma\}$ is a one-dimensional subspace. It follows from the Heine-Borel theorem that $P_\gamma(A)$ is $m$-compact in $B_\gamma$, and so it is $um$-compact in $B_\gamma$ for all $\gamma\in\Gamma$. 
Therefore, Theorem \ref{Tychonoff-Like theorem} implies that $A$ is $um$-compact in $X$.
\end{proof}
	
\begin{prop}
Let $A$ be a subset of an $m$-complete metrizable MNVL  $(X,\mathcal{M})$.	
\begin{enumerate}
			\item\label{p umcompact} If $X$ has a countable topological orthogonal system, then $A$ is sequentially $um$-compact iff $A$ is $um$-compact.
			\item\label{p mbounded} Suppose that $A$ is $m$-bounded, and $X$ has the Lebesgue property. If $A$ is $um$-compact, then $A$ is sequentially $um$-compact.
\end{enumerate}
\end{prop}	
	
\begin{proof}
\eqref{p umcompact}. It follows immediately from Proposition \ref{metrizable iff countable t.o.s}.\\
\eqref{p mbounded}. Let $x_n$ be a sequence in $A$. Find $e\in X_+$ such that $x_n$ is contained in $B_e$ (e.g., take $e=\sum_{n=1}^\infty\frac{\abs{x_n}}{2^n}$). 
Since $A$ is $um$-compact, then $A\cap B_e$ is $um$-compact in $B_e$.	
Now, since $X$ is $m$-complete and has the Lebesgue property, then $B_e$ is also $m$-complete and has the Lebesgue property. Moreover, $e$ is a quasi-interior point of $B_e$. 
Hence, by Proposition \ref{metrizable iff countable t.o.s}, the $um$-topology on $B_e$ is metrizable, consequently, $A\cap B_e$ is sequentially $um$-compact in $B_e$. 
It follows that there is a subsequence $x_{n_k}$ that $um$-converges in $B_e$ to some $x\in A\cap B_e$. 
Since $B_e$ is a projection band, then \cite[Thm.3(3)]{DEM2} implies $x_{n_k}\umc x$ in $X$. Thus, $A$ is sequentially $um$-compact.
\end{proof}

\end{document}